\documentclass[12pt,reqno]{amsart}
\usepackage{amssymb}
\usepackage{amsmath}
\usepackage{enumitem}
\usepackage{mathrsfs}
\usepackage[backref]{hyperref}
\usepackage[all]{xy}
\usepackage[margin=1.23in]{geometry}

\RequirePackage{mathrsfs} 
\usepackage[OT2,T1]{fontenc}
\DeclareSymbolFont{cyrletters}{OT2}{wncyr}{m}{n}
\DeclareMathSymbol{\Sha}{\mathalpha}{cyrletters}{"58}

\newtheorem{theorem}{Theorem}[section]
\newtheorem{lemma}[theorem]{Lemma}

\newtheorem{corollary}[theorem]{Corollary}
\newtheorem{conjecture}[theorem]{Conjecture}

\theoremstyle{definition}
\newtheorem*{ack}{Acknowledgements}
\newtheorem*{con}{Conventions}
\newtheorem{remark}[theorem]{Remark}

\numberwithin{equation}{section} \numberwithin{figure}{section}

\DeclareMathOperator{\Pic}{Pic}

\DeclareMathOperator{\Spec}{Spec}

\newcommand{\Qbar}{\overline{\QQ}}

\newcommand\PP{\mathbb{P}}

\newcommand\ZZ{\mathbb{Z}}

\newcommand\QQ{\mathbb{Q}}

\newcommand\CC{\mathbb{C}}

\newcommand\OO{\mathcal{O}}

\newcommand{\et}{\textrm{\'{e}t}}

\title[Effective Shafarevich for quartic curves]{Effectively computing integral points on the moduli of smooth quartic curves}

\author{Ariyan Javanpeykar}
\address{Ariyan Javanpeykar \\
Institut f\"{u}r Mathematik\\
Johannes Gutenberg-Universit\"{a}t Mainz\\
Staudingerweg 9, 55099 Mainz\\
Germany.}
\email{peykar@uni-mainz.de}

\subjclass[2010]
{11G35  
(14J45, 
14K30,  
14C34,  
14G40,  
14D23)}  

\keywords{  Shafarevich conjecture, Mordell conjecture, rational points, hyperelliptic locus, effectivity, quartic curves, del Pezzo surfaces}

\begin{document}
 
\begin{abstract}
We prove an effective version of the Shafarevich conjecture (as proven by Faltings) for smooth quartic curves in $\mathbb P^2$. To do so, we  establish an effective version of Scholl's finiteness result for smooth del Pezzo surfaces of degree at most four. 
\end{abstract}

\maketitle

\thispagestyle{empty}

\section{Introduction}
We show that the set of integral points on the moduli of smooth quartic hypersurfaces in $\mathbb P^2$ is finite and effectively computable.
 
In \cite{Faltings2} Faltings proved Shafarevich's conjecture for smooth proper curves: for a number field $K$, a finite set of finite places $S$ of $K$, and an integer $g\geq 2$, the set of $K$-isomorphism classes of smooth proper genus $g$ curves over $\mathcal{O}_{K,S}$ is finite. In other words, Faltings established that the set of $\mathcal{O}_{K,S}$-points of the stack of smooth proper genus $g$ curves $\mathcal M_{g}$  is finite. 

Faltings's aforementioned finiteness result for smooth proper curves over $\mathcal{O}_{K,S}$ of fixed genus is not known to be effective. That is, there is currently no algorithm that, on input given a number field $K$ and a finite set of finite places $S$ of $K$,  computes as output the finite set of $\mathcal{O}_{K,S}$-points of the stack $\mathcal M_g$. 

An effective resolution of   Shafarevich's conjecture would have deep consequences in Diophantine geometry as, for example, an effective resolution of Shafarevich's conjecture for smooth proper curves would imply an algorithmic version of the Mordell conjecture; see Section \ref{section:final} for a discussion (cf. \cite{Remo, Szpiro3, Szpirob}).

In this paper we investigate the open substack $\mathcal C_{(4;1)}$ of smooth quartic curves in the stack of smooth proper genus $3$ curves $\mathcal M_3$. More precisely, our main result says that the finite set of integral points on the open substack $\mathcal C_{(4;1)}$ of non-hyperelliptic curves of $\mathcal M_3$ can be effectively computed (see Theorem \ref{thm:quartics}). A more down-to-earth version of our result reads as follows.

\begin{theorem}\label{thm:quartics_intro}
	Let $K$ be a number field and let $S$ be a finite set of finite places of $K$. Then the set of $\mathcal{O}_{K,S}$-isomorphism classes of smooth quartic hypersurfaces in $\mathbb P^2_{\mathcal{O}_{K,S}}$ is finite and effectively computable.
\end{theorem}

Note that the analogous   statement for smooth cubic hypersurfaces in $\mathbb P^2_{\mathcal{O}_{K,S}}$ follows from effective versions of Shafarevich's finiteness theorem for elliptic curves \cite{Coates, Fuchs, Shaf1962,Silverman}.  Moreover,
 effective versions of Shafarevich's conjecture have also been obtained for cyclic curves of prime degree   \cite{dJRe, JvK, Kanel}.  The effectivity of our finiteness result   (Theorem \ref{thm:quartics_intro}) follows, as in the work of Fuchs-von K\"anel-W\"ustholz \cite{Fuchs}, von K\"anel \cite{Kanel}, and de Jong-R\'emond \cite{dJRe} from the effective resolution of the $S$-unit equation in $K$. The latter was achieved by Gy\H{o}ry-Yu   \cite{GyoryYu} (see also  \cite{EvertseGyory}) and relies on the theory of linear forms in logarithms \cite{BakerWustholz1, BakerWustholz2}.


The geometric idea behind our proof of Theorem \ref{thm:quartics_intro} is quite simple. Indeed, 
a smooth quartic curve in $\mathbb P^2$ induces a smooth del Pezzo surface of degree two by taking a double covering of $\mathbb P^2$ ramified precisely along the quartic. Moreover, the isomorphism class of the obtained smooth del Pezzo surface determines the isomorphism class of the corresponding quartic curve. This construction is a special case of what is  sometimes called the cyclic covering trick. We show that it  reduces Theorem \ref{thm:quartics_intro} to a finiteness statement about del Pezzo surfaces. 

In \cite{Scholl} Scholl proved the finiteness of all smooth del Pezzo surfaces  over a number field  $K$  with good reduction outside a fixed set of finite places $S$ of $K$. In particular, the set of $\OO_{K,S}$-isomorphism classes of smooth del Pezzo surfaces over $\OO_{K,S}$ of degree at most four is finite. To prove Theorem \ref{thm:quartics_intro}, we establish an effective version of the  latter finiteness statement.

\begin{theorem}\label{thm:DPintro}
	Let $K$ be a number field and let $S$ be a finite set of finite places.   Then the set of $\mathcal{O}_{K,S}$-isomorphism classes of smooth del Pezzo surfaces over $\mathcal{O}_{K,S}$ of degree at most $4$  is finite and effectively computable.
\end{theorem}

 

In Section \ref{section:final}  we discuss   applications of Theorem \ref{thm:quartics_intro} to an effective version of the Mordell conjecture. 
We emphasize that the results we obtain in Section \ref{section:final} form a first step towards an effective version of the Mordell conjecture, as we  prove an effective version of the Mordell conjecture for some class of complete curves, \emph{assuming}  there is an algorithm that, on input a number field $K$ and a finite set of finite places $S$ of $K$, computes as output a finite set of finite places $S'$ of $K$ containing $S$ such that all   smooth quartic curves $C$ in $\mathbb P^2_K$ with a smooth proper model  over $\OO_{K,S}$ have non-hyperelliptic reduction outside $S'$ (see Corollary \ref{cor:eff_mordell} for a precise statement). 

\begin{ack} We thank Rafael von K\"anel for inspiring discussions on the effective Shafarevich conjecture.  We are most grateful to David Holmes for many useful comments and remarks. We thank Olivier Benoist, Yuri Bilu, J\'er\'emy Blanc, Jean-Beno\^it Bost, Brian Conrad, Bas Edixhoven,  Jan-Hendrik Evertse,  Robin de Jong, Daniel Litt, Daniel Loughran, Jan-Steffen M\"uller, Bjorn Poonen,  Duco van Straten, Ronan Terpereau,  Fabio Tonini,  Angelo Vistoli, and Jonathan Wise.  We gratefully acknowledge support of SFB/Transregio 45.
\end{ack}

\begin{con}
	For $S$ a scheme, a  curve over $S$ is a flat proper finitely presented morphism $X\to S$ whose geometric fibres are connected schemes of dimension one.
We denote by $\lvert S\rvert$ the cardinality of an arbitrary finite set $S$. By $\log$ we mean the principal value of the natural logarithm. We define the product taken over the empty set as $1$.	
\end{con}

\section{Del Pezzo surfaces}


Let $k$ be an algebraically closed field. A smooth projective connected surface $X$ over $k$ is a \emph{(smooth) del Pezzo surface} if $\omega_{X/k}^{\vee}$ is ample. The degree $d(X)$ of a del Pezzo surface $X$ over $k$ is defined to be the self-intersection $(\omega_{X/k},\omega_{X,k})$ of the canonical line bundle  $\omega_{X/k}$. Note that $1\leq d(X) \leq 9$. The automorphism group of a smooth del Pezzo surface $X$ over $k$ is finite if and only if $d(X) \leq 5$. Recall that, if $d(X) \leq 7$, a line on $X$ is defined to be a $(-1)$-curve. In general, a curve $L$ on $X$ is a line if the following holds.
\begin{enumerate}
	\item If $d(X)=9$, then $L\cdot (-K_X) = 3$.
	\item If $X \cong \mathbb P^1_k \times_k \mathbb P^1_k$, then $ L\cdot (-K_X) =2$.
	\item If $d(X)\neq 9$ and $X\not\cong \mathbb P^1_k \times_k \mathbb P^1_k$, then $L\cdot (-K_X) =1$.
\end{enumerate}

Let $S$ be a scheme. Recall that a smooth proper morphism of schemes $X\to S$ is a \emph{(smooth) del Pezzo surface (over $S$)} if its geometric fibres are del Pezzo surfaces. Note that $X\to S$ is a smooth del Pezzo surface if and only if it is a Fano scheme of relative dimension two \cite[\S 2]{JLFano}. If $S$ is a connected scheme, then the degree of a del Pezzo surface $X\to S$ is constant in the fibres (this follows from \cite[Lem.~3.3]{JLFano}). 

A smooth del Pezzo surface over a field $k$ is \emph{split}, or \emph{standard}, if all   lines are defined over $k$ (see  \cite{Scholl}).


Let $X\to S$ be a del Pezzo surface over $S$.  We define $\mathcal L_{X/S} = \mathcal C_{-1,0}$ to be the Hilbert scheme of lines in $X$ \cite[\S 3.4]{Scholl}. Note that $\mathcal L_{X/S}\to S$ is a   morphism of schemes, and that $\mathcal L_{X/S}$ parametrizes the lines (i.e. exceptional curves) on $X$ over $S$.



\begin{lemma}\label{lem:good_model} Let $A$ be a principal ideal domain and let $S=\Spec (A)$.
 Let $ X\to S$ be a smooth del Pezzo surface of degree $d\leq 7$ over $S$.  If $\mathcal L_{X/S}\to S$ is constant (i.e., $\mathcal L_{X/S}$ is a disjoint union of copies of $S$), then there exist $9-d$ points $x_1,\ldots,x_{9-d}$ in $\PP^2(S)$ and an $S$-isomorphism of schemes from $X$ to the blow-up of $\mathbb P^2_S$ in $x_1,\ldots, x_{9-d}$.
\end{lemma}
\begin{proof}
This follows from \cite[Prop.~3.7]{Scholl}.
\end{proof}

Let $S$ be a Dedekind scheme (i.e., an integral noetherian normal one-dimensional scheme) with function field $K$. Let $X$ be a smooth del Pezzo surface over $K$. We say that $X$ \emph{has good reduction over $S$} if there exist a smooth del Pezzo surface $\mathcal X\to S$ over $S$ and an isomorphism $\mathcal X_K \cong X$ over $K$.  


\section{The unit equation}
Let $K$ be a number field.  For $a$ in $K$, let   $h(a)$ be the usual absolute  logarithmic Weil height of $a$,  as defined in \cite[1.6.1]{BombieriGubler}. Write $d_K = [K:\QQ]$  for the degree of $K$ over $\QQ$. Define  $D_K$ to be the absolute value of the discriminant of $K$ over $\QQ$. 

 Let $S$ be a finite set of finite places   of $K$.    Let $\mathcal O_{K,S}$ be the ring of $S$-integers in $K$.  Write  $N_S = \prod_{v\in S} N_v$ for the norm of $S$, where $N_v$  denotes the number of elements in the residue field of $v$. Also,  we let $h_S =\lvert \mathrm{Pic}(\mathcal{O}_{K,S})\rvert$ be  the class number of $\mathcal{O}_{K,S}$.

\begin{lemma}[H.W. Lenstra jr.] \label{lem:lenstra}
Let $K$ be a number field. Then 
\[ \lvert  \mathrm{Pic}(\OO_K)\rvert \leq (d_K+D_K)^{d_K} \]
\end{lemma}
\begin{proof}
This follows from a result of Lenstra \cite[Thm.~6.5]{Lenstra}. Indeed, as $\lvert\mathrm{Pic}(\OO_K)\rvert$ is the class number  of $K$ and $\frac{2}{\pi} < 1$, Lenstra's result implies that 
\[ \lvert \mathrm{Pic}(\OO_K)\rvert \leq D_K^{\frac{1}{2}} \frac{\left(d_K -1+\frac{1}{2}\log(D_K)\right)^{d_K-1}}{(d_K-1)!} \leq D_K^{\frac{1}{2}}\frac{(d_K+D_K)^{d_K}}{(d_K-1)! (d_K+D_K)}.\] This clearly implies that \[ \lvert \mathrm{Pic}(\OO_K) \rvert \leq   \left(d_K + D_K\right)^{d_K}. \qedhere\] 
\end{proof}



The following lemma is a consequence of a result of von K\"anel \cite{Kanel2} which in turn is a direct consequence of Gy\H{o}ry-Yu's theorem \cite{GyoryYu} and builds on the theory of linear forms in logarithms \cite{BakerWustholz1,BakerWustholz2}. 

\begin{lemma}[Gy\H{o}ry-Yu, von K\"anel]\label{lem:units}
	Let $K$ be a number field and let $S$ be a finite set of finite places of $K$. Let $L/K$ be a finite field extension of degree $l$ over $K$ which ramifies only over $S$. Let $S_L$ be the finite set of finite places of $L$ lying over $S$. Then the following statements hold.
	\begin{enumerate}
\item 	There is a finite set of finite places $S'_L$ of $L$ which satisfies the following properties. \begin{itemize}
			\item The finite set $S_L'$ contains $S_L$,
			\item The ring $\mathcal{O}_{L,S_L'}$ is a principal ideal domain.
\item The inequality $N_{S_L'} \leq N_{S_L} D_L^{h_{S_L}}$ holds.
\item The inequality $\lvert S_L'\rvert \leq \lvert S_L \rvert + h_{S_L}$ holds. 
			\end{itemize}
\item Let $S_L'$ be as in $(1)$. If $a$ is an element of $L$ such that $a \in \mathcal{O}_{L,S_L'}^\times$ and $1-a \in \mathcal{O}_{L,S_L'}^\times$, then the inequality 
	\[ h(a) \leq   (12 l d_K N_S D_K)^{20000 l^6  d_K^2 \lvert S\rvert     \left(l d_K + D_K^l N_S^{l} l^{l d_K \lvert S\rvert}\right)^{l d_K}   }  \]
	 holds.
\item We have
		\[d_L = ld_K, \quad D_L \leq D_K^{l} N_S^{l} l^{l d_K \lvert S\rvert}, \quad N_{S_L} \leq N_S^{l}.\]
\end{enumerate}
\end{lemma}
\begin{proof}
Note that $(1)$ follows from \cite[Lem.~4.1]{Kanel}.

Let $S_L'$ be as in $(1)$.
To prove (2) and $(3)$, we  apply \cite[Proposition 6.1.(ii)]{Kanel2} as follows. Let $S'$ be the places of $K$ lying under $S_L'$. Note that $L/K$ only ramifies over $S'$ (as $S\subset S'$). Let $m= \max(6,l) $ and $T=S'$. Also, let $U$ be the finite set of finite places of $L$ lying over $S'$ (so that $S_L'\subset U$). We now apply \emph{loc. cit.} with the above choices of $T$ and $U$ to see that  the inequality
\begin{eqnarray*}
	h(a) & \leq &  \left(2 m  d_K N_{S'}^{\log(m)}	\right)^{15md_K-1} D_K^{m-1} 
	\end{eqnarray*} holds. Note that $m\leq 6l$. In particular,
	\begin{eqnarray*}
	h(a) & \leq &  \left(12 l  d_K N_{S'}^{6l}	\right)^{90 l d_K} D_K^{6l}. 
	\end{eqnarray*} Note that $N_{S_L} \leq N_S^l$. Moreover, by our choice of $S_L'$ and $(1)$, the inequality \[N_{S'} \leq N_{S_L'} \leq  N_{S_L} D_L^{h_{S_L}} \leq N_S^l D_L^{h_{S_L}} \] holds. 
	Now, as $L$ only ramifies over $S$, it follows from Dedekind's discriminant theorem (see \cite[Lem.~6.2]{Kanel2}) that $$D_L \leq  D_K^l N_S^{l} l^{l d_K \lvert S\rvert}.$$     This concludes the proof of $(3)$. 
	By Lenstra's upper bound for the class number (Lemma \ref{lem:lenstra}), the inequality
	$$h_{S_L} \leq \lvert \Pic(\OO_L) \rvert \leq (l d_K + D_L)^{l d_K}$$ holds. We now use the upper bound $D_L \leq  D_K^l N_S^{l} l^{l d_K \lvert S\rvert}$ and obtain  
	$$h_{S_L} \leq (l d_K + D_L)^{l d_K} \leq (l d_K + D_K^l N_S^{l} l^{l d_K \lvert S\rvert})^{l d_K}.$$
Combining our inequalities, we obtain
\begin{eqnarray*}
h(a) &\leq & \left(12 l  d_K N_{S'}^{6l}	\right)^{90 l d_K} D_K^{6l} \\
&\leq & \left(12 l  d_K N_S^{6l^2} D_L^{6 l h_{S_L} } 	\right)^{90 l d_K} D_K^{6l}  \\
& \leq & \left(12 l  d_K N_S^{6l^2} \left(D_K^l N_S^{l} l^{l d_K \lvert S\rvert} \right)^{6 l h_{S_L} } 	\right)^{90 l d_K} D_K^{6l} \\
& \leq & \left(12 l  d_K N_S^{6l^2} \left(D_K^l N_S^{l} l^{l d_K \lvert S\rvert} \right)^{6 l \left((l d_K + D_K^l N_S^{l} l^{l d_K \lvert S\rvert})^{l d_K} \right) } 	\right)^{90 l d_K} D_K^{6l} \\
& \leq & (12 l d_K N_S D_K)^{6 l^2 \cdot l d_K \lvert S\rvert \cdot 6 l \left(l d_K + D_K^l N_S^{l} l^{l d_K \lvert S\rvert}\right)^{l d_K}  \cdot 90 l d_K \cdot 6l} \\
& \leq &  (12 l d_K N_S D_K)^{20000 l^6  d_K^2 \lvert S\rvert     \left(l d_K + D_K^l N_S^{l} l^{l d_K \lvert S\rvert}\right)^{l d_K}   }  .
\end{eqnarray*} This concludes the proof of $(2)$.  
\end{proof}

\section{An  effective Shafarevich theorem  for del Pezzo surfaces} 
We now prove an effective version of Scholl's finiteness theorem for   del Pezzo surfaces of degree at most four. As in Scholl's paper \cite{Scholl}, we reduce to a finiteness statement about the unit equation.

The unit equation appears via a consideration of the coordinates of certain points on $\mathbb P^2$ which are in general position. Here by general position, we
mean the following. Let $1\leq d\leq 9$ and let $P_1, \ldots, P_{9-d}$ be points on $\mathbb P^2_k$, where $k$ is a field. We say that these points are (geometrically) in general position if no three lie on a line in $\mathbb P^2_{\overline{k}}$, no six lie on a conic in $\mathbb P^2_{\overline{k}}$, and no eight lie on a cubic in $\mathbb P^2_{\overline{k}}$
which is singular at one of the points.

\begin{theorem}\label{thm:DP}
	Let $K$ be a number field, let $S$ be a finite set of finite places of $K$, and let $1\leq d\leq 4$. Let $X$ be a   smooth del Pezzo surface of degree $  d $ over $K$ with good reduction outside $S$. There exist a finite field extension $L/K$ of degree at most $240!$ which ramifies only over $S$,  and there exist $5-d$ points 
	$$P_5 = (a_5:b_5:1),\ldots,P_{9-d} = (a_{9-d}:b_{9-d}:1)$$ in $\mathbb P^2(L)$  
	 such that  the following statements hold.
	\begin{enumerate}
	\item The  surface $X_L$ is the blow-up of $\mathbb P^2_L$ in \[(0:0:1), (0:1:0), (1:0:0), (1:1:1), P_5, \ldots, P_{9-d},  \] and these points are in general position.
	\item Define $l_d$ to be the order of the Weyl group of the root system of $E_{9-d}$. Then the inequality
		\[  \max_{i=5, \ldots, 9-d}\left(h(a_i),  h(b_i)\right) \leq (12 l_d d_K N_S D_K)^{20000 l_d^6  d_K^2 \lvert S\rvert     \left(l_d d_K + D_K^{l_d} N_S^{l_d} l_d^{l_d d_K \lvert S\rvert}\right)^{l_d d_K}   }  \] holds.  
		\item  Write $S_L$ for the set of finite places of $L$ lying over $S$. Then 
		\[d_L \leq l_dd_K, \quad D_L \leq D_K^{l_d} N_S^{l_d} l_d^{l_d d_K \lvert S\rvert}, \quad N_{S_L} \leq N_S^{l_d}.\]
	\end{enumerate}
\end{theorem}
 \begin{proof} Let $L$ be the smallest number field such that all lines of $X$ are defined over $L$. Note that the number field $L$ is of degree at most $l_d$ over $K$ and ramifies only over $S$. (This follows from Scholl's \cite[Prop.~3.6]{Scholl} and standard facts about lines on smooth del Pezzo surfaces \cite[\S 8.2]{Dol12}.)
 Let $S_L$ be the finite set of finite places of $L$ lying over $S$.
 
By the first part of Lemma \ref{lem:units}, there exists a finite set $S_L'$ of finite places of $L$ containing $S_L$ such that the following statements hold. 
\begin{itemize}
\item The ring $\mathcal{O}_{L,S_L'}$ is a principal ideal domain.
\item The inequality $N_{S_L'} \leq N_{S_L} D_L^{h_{S_L}}$ holds.
\item The inequality $\lvert S_L'\rvert \leq \lvert S_L \rvert + h_{S_L}$ holds.
\end{itemize}

To prove the theorem, we now follow the proof of \cite[Prop.~4.2]{Scholl}. 
Since $X$ has good reduction outside $S$, we see that $X_L$ has good reduction outside $S_L$ (and thus $S_L'$). Let $\mathcal X\to \Spec \mathcal O_{L,S_L'}$ be a smooth del Pezzo surface such that $\mathcal X_L$ is isomorphic to $X_L$. Since $\OO_{L,S_L'}$ is a principal ideal domain and $\mathcal L_{\mathcal X/\OO_{L,S_L'}}$ is constant over $\Spec \mathcal{O}_{L,S_L'}$, it follows from Lemma \ref{lem:good_model} that  there are points $P_5, \ldots, P_{9-d}$ in $\mathbb P^2(\OO_{L,S_L'})$ such that   $\mathcal{X}$ is   isomorphic to the blow-up of $\mathbb P^2_{\OO_{L,S_L'}}$ in  the $\mathcal{O}_{L,S_L'}$-points \[(0:0:1), (0:1:0), (1:0:0), (1:1:1), P_5, \ldots, P_{9-d}.\]  In particular, as $\mathcal X$ is smooth over $\OO_{K,S_L'}$, the points \[(0:0:1), (0:1:0), (1:0:0), (1:1:1), P_5, \ldots, P_{9-d},  \] are in general position.
 For $i=5,\ldots, 9-d$, write $P_i = (\alpha_i:\beta_i:\gamma_i)$ with $\alpha_i,\beta_i,\gamma_i \in \mathcal{O}_{L,S_L'}$. Now, as no three of any collection of four of these points   are collinear, we see that $\alpha_i \beta_i \gamma_i \in \mathcal {O}_{L,S_L'}^\times$. For $i=5,\ldots, 9-d$, we define $a_i = \frac{\alpha_i}{\gamma_i}$, and $b_i = \frac{\beta_i}{\gamma_i}$. Now,  \[(0:0:1), (0:1:0), (1:0:0), (1:1:1), P_5, \ldots, P_{9-d},  \] are in general position and $\mathcal{X}$ is   isomorphic to the blow-up of $\mathbb P^2_{\OO_{L,S_L'}}$ in  these $\mathcal{O}_{L,S_L'}$-points. Again, as no three of any collection of four of these points are collinear, it follows that $(1- a_i)(1-b_i) \in \mathcal {O}_{L,S_L'}^\times$. Therefore, for all $i=5,\ldots, 9-d$, the algebraic numbers $a_i$ and $b_i$ are solutions to the $S_L'$-unit equation in $L$. In particular, by the second part of Lemma \ref{lem:units}, the inequality 
 \begin{eqnarray*}
  \max_{i=5, \ldots, 9-d}\left(h(a_i),  h(b_i)\right)
 & \leq & (12 l_d d_K N_S D_K)^{20000 l_d^6  d_K^2 \lvert S\rvert     \left(l_d d_K + D_K^{l_d} N_S^{l_d} l_d^{l_d d_K \lvert S\rvert}\right)^{l_d d_K}   }
  \end{eqnarray*}  holds. This concludes the proof of $(1)$ and $(2)$. Note that $(3)$ follows from the third part of Lemma \ref{lem:units}.
\end{proof}

\begin{remark}
	We note that, with notation as in $(2)$ of Theorem \ref{thm:DP}, 
	\[l_1 = 696729600, \quad l_2 = 2903040 , \quad l_3 = 51840, \quad l_4 = 1920; \] see \cite[Cor.~8.2.16]{Dol12}.
\end{remark}


\begin{proof}[Proof of Theorem \ref{thm:DPintro}]
	Let $K_{\mathrm{split}}$ be the compositum of all number fields $L/K$ such that $L$ is ramified only over $S$ and of degree at most $240!$. Then, by Theorem \ref{thm:DP}, any smooth del Pezzo surface over $K$ of degree $1\leq d\leq 4$ with good reduction outside $S$ is split  over the number field $K_{\mathrm{split}}$. It follows from Theorem \ref{thm:DP} that the set of $K_{\mathrm{split}}$-isomorphism classes of smooth del Pezzo surfaces over $K$ of degree $1\leq d\leq 4$ with good reduction outside $S$ is finite and effectively computable. (Indeed, the heights of the coefficients of the coordinates of the points we require to blow-up in $\mathbb P^2({K_{\mathrm{split}}})$ are bounded explicitly.) As the automorphism group of a smooth del Pezzo surface of degree at most four is finite (and effectively computable using methods as in \cite{Blanc2, Blanc4}), a standard Galois cohomological argument now concludes the proof (see part (a) of the proof of \cite[Thm.~4.5]{Scholl}).  
\end{proof}

\section{An effective Shafarevich   theorem  for smooth  quartic curves}
Let $\mathcal H_3$ be the stack of hyperelliptic curves in $\mathcal M_3$. The morphism $\mathcal H_3\to \mathcal M_3$ is a closed immersion and  the coarse moduli space of $\mathcal H_3$ is affine.

Let $\mathcal M_3^{\textrm{nh}}$ be the complement of $\mathcal H_3$ in $\mathcal M_3$. Note that $\mathcal M_3^{\textrm{nh}}$ parametrizes smooth proper non-hyperelliptic curves of genus $3$, and that $\mathcal M_3^{\textrm{nh}}$ is an open substack of $\mathcal M_3$.

Let $\mathrm{Hilb}$ be the Hilbert scheme of smooth quartic curves in $\PP^2$, i.e., $\mathrm{Hilb}$ is the affine scheme over $\ZZ$ given by the complement of the discriminant divisor in $\mathbb P(\mathrm{H}^0(\mathbb P^2_{\mathbb Z}, \mathcal O_{\mathbb P^2_{\mathbb Z}}(4)))$. Note that $\mathrm{PGL}_{3,\ZZ}$ acts on $\mathrm{Hilb}$.
Let $\mathcal C_{(4;1)}:=[\mathrm{PGL}_{3,\ZZ}\backslash \mathrm{Hilb}]$ be the stack of smooth quartic curves in $\PP^2$  (see \cite{Ben13, JL}).  Note that the natural morphism from $\mathcal C_{(4;1)}$ to the  complement $\mathcal M_3^{\textrm{nh}}$ of $\mathcal H_3$ in $\mathcal M_3$ is an isomorphism of stacks over $\ZZ$. (The only subtle point here is that the automorphism group of a smooth quartic curve in $\mathbb P^2$ equals its ``linear''  automorphism group as a hypersurface; see \cite{Chang, Poo05}.)

If $S$ is a scheme, then  an $S$-object of $\mathcal C_{(4;1)}$ is not necessarily isomorphic to a smooth quartic hypersurface in $\mathbb P^2_S$. There are obstructions (coming from the Brauer group of $S$ and $\mathrm{H}^1_{\et}(S, \mathrm{GL}_{3,S})$ \cite[\S 2.1.2]{JL}) to an $S$-object of $\mathcal C_{(4;1)}$ being a smooth quartic curve in $\mathbb P^2_S$ (and not only in a non-trivial Brauer-Severi scheme or projective bundle over $S$). On the other hand, if $k$ is a field, then any $k$-object of $\mathcal C_{(4;1)}(k)$ is in fact a smooth quartic curve in $\mathbb P^2_k$ (cf. \cite{Lorenzotwists}). 


   We now prove  an effective version of Faltings's theorem (\emph{quondam} Shafarevich's conjecture)  for smooth quartic curves (i.e., non-hyperelliptic smooth proper genus $3$ curves).

\begin{theorem}\label{thm:quartics}
	Let $K$ be a number field and let $S$ be a finite set of finite places of $K$. The essential image of the functor $\mathcal C_{(4;1)}(\mathcal{O}_{K,S})\to \mathcal C_{(4;1)}(K)$  is finite and effectively computable.
\end{theorem}
\begin{proof}	 
	The set of isomorphism classes of rank three vector bundles on $\Spec \OO_{K,S}$ is finite and effectively computable, as it is given by $\mathrm{H}^1_{\et}(\OO_{K,S},\mathrm{GL}_{3,\OO_{K,S}})$. (The latter cohomology set can be computed explicitly as follows.  Firstly,  Borel's finiteness theorem \cite[Thm.~5.1]{Borel2} is effective. Therefore, the set $c(\OO_{K,S}, \mathrm{GL}_{3,\OO_{K,S}})$ (with notation as in \cite[\S 5]{GilleMoretBailly}) is finite and effectively computable. Finally, there is a natural bijection $c(\OO_{K,S}, \mathrm{GL}_{3,\OO_{K,S}})  \cong \mathrm{H}^1_{\et}(\OO_{K,S}, \mathrm{GL}_{3,\OO_{K,S}})$.)

	 Let $S'$ be a finite set of finite places of $K$ with the following properties.
	\begin{enumerate}
		\item All elements of $\mathrm{H}^1_{\et}(\OO_{K,S},\mathrm{GL}_{3,\OO_{K,S}})$  are trivial over $\Spec \OO_{K,S'}$,
		\item the set $S'$ contains all places lying over $2$, and
		\item $\mathrm{Pic}(\OO_{K,S'})  =0$. 
	\end{enumerate} Note that, by Lemma \ref{lem:units} and the fact that $\mathrm{H}^1_{\et}(\OO_{K,S},\mathrm{GL}_{3,\OO_{K,S}})$ is effectively computable,   we can indeed effectively determine such a finite set  $S'$.
	
	Note that the set of $K$-isomorphism classes of smooth del Pezzo surfaces of degree two over $K$ with good reduction outside $S'$ is finite and effectively computable (Theorem \ref{thm:DPintro}).  Let $D_1, \ldots, D_n$ be smooth del Pezzo surfaces of degree two over $K$ such that any smooth del Pezzo surface of degree two over $K$ with good reduction outside $S'$ is isomorphic to some $D_i$ with $i$ in $\{1,\ldots,n\}$. 
	
	Note that, all objects of $\mathcal C_{(4;1)}(K)$ are smooth quartic curves in $\mathbb P^2_K$, 
	all vector bundles of rank three over $\OO_{K,S}$ trivialize over $\OO_{K,S'}$, and the Picard group of $\Spec \OO_{K,S'}$ is trivial. Therefore,   if $Y$ is in the essential image of the functor $\mathcal C_{(4;1)}(\OO_{K,S})\to \mathcal C_{(4;1)}(K)$, then   there is a smooth quartic $\mathcal Y$ in $\mathbb P^2_{\OO_{K,S'}} $ whose generic fibre $\mathcal Y_K$ is $K$-isomorphic to $Y$ (use \cite[Lem.~4.8.(2)]{JL} and the explicit description of the functor of points of $\mathcal C_{(4;1)}$ given in \cite[\S 2.3.2]{BenoistThesis}).
	
	  Let $f \in \OO_{K,S'}[x_0,x_1,x_2]_4$ be a homogeneous polynomial such that $\mathcal Y$ is isomorphic to the zero locus  of $f$ in $\mathbb P^2_{\OO_{K,S'}}$. Since a double cover $\mathcal D$ of $\mathbb P^2_{\OO_{K,S'}}$ ramified along $\mathcal Y$ can   be written as the zero locus of $x_3^2 = f$ in a suitable weighted projective space, it follows from the Jacobian criterion for smoothness that $\mathcal D$ is smooth over $\OO_{K,S'}$ (here we use that $S'$ contains all the places lying over $2$). Therefore,  a double cover $ \mathcal  D$ of $\mathbb P^2_{\OO_{K,S'}}$ ramified precisely along $\mathcal Y$ is a smooth del Pezzo surface of degree two over $\OO_{K,S'}$. Note that the isomorphism class of $\mathcal D_{\overline{K}}$ determines the isomorphism class of $\mathcal Y_{\overline{K}}$.
	 
	  There exists an integer   $i$ in $\{1,\ldots,n\}$ such that $  \mathcal D_K$ is $K$-isomorphic to $D_i$. Therefore, as the automorphism group of a smooth proper genus three curve is finite, up to a standard Galois cohomological argument, we conclude that the set of isomorphism classes of $K$-objects of $\mathcal C_{(4;1)}(K)$ which come from an $\OO_{K,S}$-point of $\mathcal C_{(4;1)}$ is finite and effectively computable.
\end{proof}

 \begin{remark} To explain the idea  behind our proof of Theorem \ref{thm:quartics}, let $\mathcal {DP}_2$ be the stack of smooth del Pezzo surfaces of degree two over $\ZZ$. Note that $\mathcal {DP}_2$ is a Deligne-Mumford separated algebraic stack of finite type over $\ZZ$. The cyclic covering trick exhibits $\mathcal {DP}_{2,\ZZ[1/2]}$ as a (non-neutral) $ \mu_2$-gerbe $\mathcal {DP}_{2,\ZZ[1/2]} \to \mathcal C_{(4;1), \ZZ[1/2]}$ over the stack $\mathcal C_{(4;1),\ZZ[1/2]}$. More precisely, given a smooth del Pezzo surface $X$ over a ring $A$ with $2\in A^\times$, the anti-canonical map exhibits $X$ as a double cover of some (twisted) projective space of relative dimension two over $A$. The branch locus of this double cover is a (twisted) smooth quartic curve over $A$. We refer the reader to \cite{ArsieVistoli} for a further discussion  of the stack of smooth del Pezzo surfaces of degree two as a certain stack of cyclic covers.
 
 It is the structure of $\mathcal {DP}_{2,\ZZ[1/2]}$ as a $\mu_2$-gerbe over $\mathcal C_{(4;1),\ZZ[1/2]}$ that we have exploited in our proof of Theorem \ref{thm:quartics}. 
 
 	It seems worthwile noting that the stacks $\mathcal {DP}_2$ and $\mathcal C_{(4;1)}$ are not isomorphic (not even over $\CC$). Indeed, all smooth del Pezzo surfaces of degree two over $\CC$ have a non-trivial automorphism of order two, whereas the general quartic curve in $\mathbb P^2_{\CC}$ has no non-trivial automorphisms. 
 	
 	On the other hand, the induced morphism on coarse moduli spaces   $\mathcal {DP}_{2,\CC}^{\textrm{coarse}}\to \mathcal C^{\textrm{coarse}}_{(4;1),\CC}$ is an isomorphism of complex algebraic affine varieties   \cite[\S7.2]{Looijenga}.
 \end{remark}

 \begin{corollary}\label{cor:quartics}
 	Let $K$ be a number field and let $S$ be a finite set of finite places of $K$. The set of isomorphism classes of the groupoid $\mathcal C_{(4;1)}(\mathcal{O}_{K,S})$ is finite and effectively computable.
 \end{corollary}
 \begin{proof}
 As the functor $\mathcal C_{(4;1)}(\OO_{K,S})\to \mathcal C_{(4;1)}(K)$ is injective on the underlying sets of isomorphism classes, this follows from Theorem \ref{thm:quartics}.
 \end{proof}

\begin{proof}[Proof of Theorem \ref{thm:quartics_intro}] As the set of $\OO_{K,S}$-isomorphism classes of smooth quartic curves over $\OO_{K,S}$ is a subset of the set of isomorphism classes of the groupoid $\mathcal C_{(4;1)}(\OO_{K,S})$, this follows from Corollary \ref{cor:quartics}.
	\end{proof}

\begin{remark}
	Note that the Shafarevich conjecture is effective for hyperelliptic curves \cite{dJRe, JvK, Kanel}. In particular, the finite set of integral points on $\mathcal H_3$ can be computed effectively.
	
	 Nonetheless, we are not able to infer an effective version of the Shafarevich conjecture for all smooth proper genus three curves by combining the results of \emph{loc. cit.} with  Theorems \ref{thm:quartics_intro} and \ref{thm:quartics}.   On the other hand, we are able to reduce the effective Mordell conjecture for some complete curves to a statement about effectively bounding primes of hyperelliptic reduction on a smooth quartic curve over a fixed  $\mathcal{O}_{K,S}$ (see Theorem \ref{thm:eff_mordell2}). 
	 
	 An effective version of the Shafarevich conjecture for all smooth proper genus three curves would imply an effective version of the Mordell conjecture for some class of curves; see Section \ref{section:final}.
\end{remark}


\section{Towards Algorithmic Mordell for some class of curves}\label{section:final}

We give   an application of our main result on  computing integral points on the stack $\mathcal C_{(4;1)} = \mathcal M_3^{\textrm{nh}}$ (Theorem \ref{thm:quartics}) to  computing rational points on certain complete hyperbolic curves.

Our aim is to provide a criterion for a complete curve to satisfy a version of the Mordell conjecture in which one can also algorithmically determine the set of rational points; see Corollary \ref{cor:eff_mordell} for a precise statement. Let us start with stating this conjecture.

\begin{conjecture}[Algorithmic Mordell]\label{conj}
	There exists an algorithm that, on input given a number field $K$, a smooth projective geometrically connected curve $X$ over $K$ of genus at least two, and a number field $L$ over $K$, computes as output the finite set $X(L)$. 
\end{conjecture}

To make this conjecture mathematically precise, let us note that by ``algorithm'' we mean Turing machine (as defined in \cite{HU}). We refer the reader to \cite{PoonenRat} for a useful discussion.

The finiteness of the set $X(L)$ follows from Faltings's theorem (\emph{quondam} Mordell's conjecture) \cite{Faltings2, Szpiroa}.  
Note that before Faltings's theorem, there was not a single example known of a number field $K$ and a smooth proper geometrically connected curve of genus at least two over $K$ such that, for all number fields $L$ over $K$, the set $X(L)$ was provably finite. 

Conjecture \ref{conj} is clearly different from \textit{Mordell effectif} as stated in \cite[\S 4]{Moret-Bailly3}. Indeed, Conjecture \ref{conj}  implies  that the height of an $L$-rational point of $X$ is bounded by some effectively computable real number depending only on $X$, $K$, and $L$, but it does not infer any linear or polynomial dependence on   the discriminant of $L$ nor the height of $X$.

Faltings's proof of the Mordell conjecture exploits the fact that, for all  smooth projective curves $Y$ over $\CC$ of genus at least two, there exist a finite \'etale cover $X\to Y$, an integer $g\geq 3$, and a finite morphism $X\to \mathcal M_{g,\CC}$; see \cite{MartinDeschamps, Szpiroa}. An effective version of the Shafarevich conjecture for \emph{all} smooth proper genus $g$ curves would therefore imply Conjecture \ref{conj}; see \cite{Remo} for a precise statement.   

Levin has shown that effective Shafarevich theorems for ``special'' classes of curves have applications to effectively computing  integral points on affine curves \cite{Levin1}. The main result of this section (see Theorem \ref{thm:eff_mordell2} and Corollary \ref{cor:eff_mordell} below) aims at showing a similar result for rational points on complete curves $X$ mapping non-trivially to $\mathcal M_3$.


In \cite{Zaal} Zaal has explicitly constructed  complete curves $X$ which map finitely to $\mathcal M_3$. 
The aim of this section is to prove Algorithmic Mordell for such complete curves $X$, under suitable assumptions.

\begin{theorem}\label{thm:eff_mordell2} Let $K$ be a number field and let $S$ be a finite set of finite places of $K$. Let $X$ be a complete curve over $\mathcal{O}_{K,S}$ which maps finitely to $\mathcal M_{3,\mathcal{O}_{K,S}}$.
	
	Suppose that there is an effectively computable  finite set of finite places $S'$ of $K$ containing $S$ such that all smooth quartic curves in the image of $X(K)$ in $\mathcal M_3(K)$   have no hyperelliptic reduction over $\mathcal{O}_{K,S'}$.
	
	Then the set $X(K)$ is finite and effectively computable. 
\end{theorem}
\begin{proof} 
	By assumption, we are given explicitly a finite  morphism $p:X\to \mathcal M_{3,\mathcal{O}_{K,S}}$, where $\mathcal M_3$ is the stack of smooth proper genus three curves. 
	
	Let $\mathcal X'$ be the image of $p$ in $\mathcal M_{3,\mathcal{O}_{K,S}}$. Note that $ \mathcal X'$ is closed in $\mathcal M_{3,\mathcal{O}_{K,S}}$ and proper over $\mathcal{O}_{K,S}$.
	As $X_K$ is proper and $\mathcal H_{3,K}$ has affine coarse space, the intersection $\mathcal X'_{K}\times_{\mathcal M_{3,K} } \mathcal H_{3,K}$  of $\mathcal X_K'$ and $H_X:=\mathcal H_{3,K}$ in $\mathcal M_{3,K}$ is a  finite scheme over $K$. To prove the theorem, we may and do assume that all closed points of $H_X$ lie in $\mathcal M_{3}(K)$.

	Let $S'$ be as in the statement of the theorem.
	By Theorem \ref{thm:quartics} the essential image of the functor $$\mathcal C_{(4;1)}(\mathcal{O}_{K,S'}) \to \mathcal C_{(4;1)}(K)$$ is finite and effectively computable. Let $y_1,\ldots, y_n$ in $\mathcal C_{(4;1)}(K)$ be representatives for the essential image of $\mathcal C_{(4;1)}(\mathcal{O}_{K,S'}) \to \mathcal C_{(4;1)}(K).$ 
	
	By the defining property of $S'$,   the image of  $X(K) \to \mathcal M_3(K)$ is contained  in the union of $H_X$ with the image of $\mathcal C_{(4;1)}(\mathcal{O}_{K,S'}) \to \mathcal C_{(4;1)}(K)\to \mathcal M_{3}(K)$. In particular, the  image of $X(K)$ in $\mathcal M_3(K)$ lies in the effectively computable finite set of $K$-rational points $H_X \cup \{y_1,\ldots,y_n\}$ of $\mathcal M_{3}(K)$. 
	As $X_K\to \mathcal M_{3,K}$ is a finite morphism and the image of $X(K)$ is finite, we can now conclude that $X(K)$ is finite.


	Let $y$ be  in $H_X\cup \{y_1,\ldots,y_n\} \subset \mathcal M_3(K)$. Note that the fibre of $X\to \mathcal M_{3,K}$ over $y$ with respect to $X_K\to \mathcal M_{3,K}$ is either empty or a zero-dimensional subscheme of $X$. The theory of Gr\"obner bases allows one to effectively compute whether the fibre over $y$ is empty. (Here we use that the morphism $p$ is given explicitly, so that one can effectively compute equations for the closed subscheme $p^{-1}(y_i)$.) In particular,  to conclude the proof, we may and do assume that the fibre over $y$ is non-empty. Now, to conclude the computation of $X(K)$, it suffices to show that the set of points on a non-empty zero-dimensional finite scheme over $\Qbar$ can be effectively computed. This can be done using elimination theory and factoring of polynomials over number fields \cite{L, LLL}. 
\end{proof}

\begin{remark}
	Let us discuss the assumption in Theorem \ref{thm:eff_mordell2}. To do so, let $K$ be a number field, let $S$ be a finite set of finite places of $K$, and let $X\to \Spec \OO_{K,S}$ be a smooth proper curve of genus $3$.  If $X_K$ is not a hyperelliptic curve, then the set of primes $\mathfrak p\subset \OO_{K,S}$ such that the fibre of $X$ over $\mathfrak{p}$ is hyperelliptic is finite. Moreover, if $X$ is fixed, then the set of ``hyperelliptic reductions'' of $X$ is   effectively computable, as it is given by the intersection product of the hyperelliptic locus in $\mathcal M_3$ with the $\OO_{K,S}$-section of $\mathcal M_3$ corresponding to $X$ (pulled-back to $\OO_{K,S}$). Now, the hypothesis in Theorem \ref{thm:eff_mordell2} says  that, for all number fields $K$ and all finite sets of finite places $S$ of $K$, we can effectively compute a finite set of finite places $S'$ of $K$ such that \textit{all} smooth proper genus $3$ curves $X$ over $\OO_{K,S}$ with $X_K$ not a hyperelliptic curve have non-hyperelliptic reduction at all $\mathfrak{p}\not\in S'$. 
	
	Note that  the Shafarevich conjecture (as proven by Faltings) implies that the set $S'$ exists (but might not be effectively computable). However, an effective version of the Shafarevich conjecture for genus three curves would imply that this assumption holds. On the other hand, the assumption made in Theorem \ref{thm:eff_mordell2} is a priori weaker than an effective version of the Shafarevich conjecture.
\end{remark}

\begin{corollary}\label{cor:eff_mordell}    Let $X$ be a proper curve over a number field $K$ and let $X\to \mathcal M_{3,K}$ be a finite morphism.
	
	Suppose  there is an algorithm that, on input a number field $L$ over $K$, computes as output  a finite set of finite places $S$ of $L$ such that all smooth quartic curves $C$ in the image of $X(L) \to \mathcal M_3(L)$  lie in the essential image of $\mathcal M_3^{\textrm{nh}}(\OO_{L,S})\to \mathcal M_3(L)$. 
	
	Then Algorithmic Mordell holds for $X$, i.e.,  there is an algorithm that, on input a number field $L$, computes as output the finite set $X(L)$.
\end{corollary}
\begin{proof}
This follows directly from Theorem \ref{thm:eff_mordell2}.
\end{proof}


\begin{remark}
	Note that the Bogomolov-Miyaoka-Yau inequality implies that a smooth proper genus two curve over $\CC$ does not map finitely to $\mathcal M_{3,\CC}$; see \cite{Kotschick}. It seems reasonable to suspect that there are complete hyperbolic  complex algebraic curves  $X$ which do not map finitely to $\mathcal M_{3,\CC}$, even after passing to a finite \'etale cover.    On the other hand, any complete hyperbolic curve maps finitely, up to a finite \'etale cover, to  $\mathcal M_{g,\CC}$  for some $g\geq 3$; see \cite{MartinDeschamps}.
\end{remark}


\begin{remark} 
	We emphasize that our proof of Theorem \ref{thm:eff_mordell2} gives a non-efficient algorithm for several reasons. For instance, part of the algorithm consists of writing down all solutions to the unit equation in some number ring. Moreover, one has to work with number fields of high degree.
\end{remark}

\begin{corollary}\label{cor}   Let $X\to Y$ be a finite \'etale morphism of smooth proper geometrically connected curves over a number field $K$. Let $X\to \mathcal M_{3,K}$ be a finite morphism.

Suppose  there is an algorithm that, on input a number field $L$ over $K$,  computes as output  a finite set of finite places $S$ of $L$ such that all smooth quartic curves $C$ in the image of $X(L) \to \mathcal M_3(L)$   have non-singular non-hyperelliptic reduction over $\OO_{L,S}$. 

Then there is an algorithm that, on input a number field $L$, computes the finite set $Y(L)$.
\end{corollary}
\begin{proof}
	This follows   from Corollary \ref{cor:eff_mordell} and the quantitative version of the Chevalley-Weil theorem for curves \cite{BiluCW}. (Note that a smooth quartic curve over $L$ has non-singular non-hyperelliptic reduction over $\OO_{L,S}$ if and only if it lies in the essential image of $\mathcal M_3^{\textrm{nh}}(\OO_{L,S})\to \mathcal M_3(L).$ In other words,   a smooth quartic curve $X$ over $K$ has non-singular non-hyperelliptic reduction  over $\OO_{L,S}$ if, and only if, its minimal regular proper model $\mathcal X\to \Spec \OO_{L,S}$ (as defined in \cite[Defn.~9.3.12]{Liu2}) is smooth over $\OO_{L,S}$ and, for all $b$ in $\Spec \OO_{L,S}$, the fibre $\mathcal X_b$ is a non-hyperelliptic curve.)
	\end{proof}


\bibliography{refsci}{}

\def\cprime{$'$}
\begin{thebibliography}{10}

\bibitem{ArsieVistoli}
A~Arsie and A.~Vistoli.
\newblock Stacks of cyclic covers of projective spaces.
\newblock {\em Compos. Math.}, 140(3):647--666, 2004.

\bibitem{BakerWustholz1}
A.~Baker and G.~W{\"u}stholz.
\newblock Logarithmic forms and group varieties.
\newblock {\em J. Reine Angew. Math.}, 442:19--62, 1993.

\bibitem{BakerWustholz2}
A.~Baker and G.~W{\"u}stholz.
\newblock {\em Logarithmic forms and {D}iophantine geometry}, volume~9 of {\em
  New Mathematical Monographs}.
\newblock Cambridge University Press, Cambridge, 2007.

\bibitem{BenoistThesis}
O.~Benoist.
\newblock Espace de modules d'intersections compl\`etes lisses.
\newblock {\em Ph.D. thesis}.

\bibitem{Ben13}
O.~Benoist.
\newblock S\'eparation et propri\'et\'e de {D}eligne-{M}umford des champs de
  modules d'intersections compl\`etes lisses.
\newblock {\em J. Lond. Math. Soc. (2)}, 87(1):138--156, 2013.

\bibitem{BiluCW}
Y.~Bilu, M.~Strambi, and A.~Surroca.
\newblock Quantitative {C}hevalley-{W}eil theorem for curves.
\newblock {\em Monatsh. Math.}, 171(1):1--32, 2013.

\bibitem{Blanc2}
J.~Blanc.
\newblock Finite abelian subgroups of the {C}remona group of the plane.
\newblock {\em C. R. Math. Acad. Sci. Paris}, 344(1):21--26, 2007.

\bibitem{Blanc4}
J.~Blanc.
\newblock Elements and cyclic subgroups of finite order of the {C}remona group.
\newblock {\em Comment. Math. Helv.}, 86(2):469--497, 2011.

\bibitem{BombieriGubler}
E.~Bombieri and W.~Gubler.
\newblock {\em Heights in {D}iophantine geometry}, volume~4 of {\em New
  Mathematical Monographs}.
\newblock Cambridge University Press, Cambridge, 2006.

\bibitem{Borel2}
A.~Borel.
\newblock Some finiteness properties of adele groups over number fields.
\newblock {\em Inst. Hautes \'Etudes Sci. Publ. Math.}, (16):5--30, 1963.

\bibitem{Chang}
H.~C. Chang.
\newblock On plane algebraic curves.
\newblock {\em Chinese J. Math.}, 6(2):185--189, 1978.

\bibitem{Coates}
J.~Coates.
\newblock An effective {$p$}-adic analogue of a theorem of {T}hue. {III}. {T}he
  diophantine equation {$y\sp{2}=x\sp{3}+k$}.
\newblock {\em Acta Arith.}, 16:425--435, 1969/1970.

\bibitem{dJRe}
R.~de~Jong and G.~R{\'e}mond.
\newblock Conjecture de {S}hafarevitch effective pour les rev\^etements
  cycliques.
\newblock {\em Algebra Number Theory}, 5(8):1133--1143, 2011.

\bibitem{Dol12}
I.~V. Dolgachev.
\newblock {\em Classical algebraic geometry}.
\newblock Cambridge University Press, Cambridge, 2012.
\newblock A modern view.

\bibitem{EvertseGyory}
J.-H. Evertse and K.~Gy{\H{o}}ry.
\newblock Effective results for unit equations over finitely generated integral
  domains.
\newblock {\em Math. Proc. Cambridge Philos. Soc.}, 154(2):351--380, 2013.

\bibitem{Faltings2}
G.~Faltings.
\newblock Endlichkeitss\"atze f\"ur abelsche {V}ariet\"aten \"uber
  {Z}ahlk\"orpern.
\newblock {\em Invent. Math.}, 73(3):349--366, 1983.

\bibitem{Fuchs}
C.~Fuchs, R.~von K{\"a}nel, and G.~W{\"u}stholz.
\newblock An effective {S}hafarevich theorem for elliptic curves.
\newblock {\em Acta Arith.}, 148(2):189--203, 2011.

\bibitem{GilleMoretBailly}
P.~Gille and L.~Moret-Bailly.
\newblock Actions alg\'ebriques de groupes arithm\'etiques.
\newblock In {\em Torsors, \'etale homotopy and applications to rational
  points}, volume 405 of {\em London Math. Soc. Lecture Note Ser.}, pages
  231--249. Cambridge Univ. Press, Cambridge, 2013.

\bibitem{GyoryYu}
K.~Gy{\H{o}}ry and K.~Yu.
\newblock Bounds for the solutions of {$S$}-unit equations and decomposable
  form equations.
\newblock {\em Acta Arith.}, 123(1):9--41, 2006.

\bibitem{HU}
J.~E. Hopcroft and J.~D. Ullman.
\newblock {\em Formal languages and their relation to automata}.
\newblock Addison-Wesley Publishing Co., Reading, Mass.-London-Don Mills, Ont.,
  1969.

\bibitem{JL}
A.~Javanpeykar and D.~Loughran.
\newblock Complete intersections: {M}oduli, {T}orelli, and good reduction.
\newblock {\em arXiv:1505.02249, To appear in Math. Annalen}.

\bibitem{JLFano}
A.~Javanpeykar and D.~Loughran.
\newblock Good reduction of {F}ano threefolds and sextic surfaces.
\newblock {\em arXiv:1512.07189}.

\bibitem{JvK}
A.~Javanpeykar and R.~von K{\"a}nel.
\newblock Szpiro's small points conjecture for cyclic covers.
\newblock {\em Doc. Math.}, 19:1085--1103, 2014.

\bibitem{Kotschick}
D.~Kotschick.
\newblock Signatures, monopoles and mapping class groups.
\newblock {\em Math. Res. Lett.}, 5(1-2):227--234, 1998.

\bibitem{L}
A.~K. Lenstra.
\newblock Factoring polynomials over algebraic number fields.
\newblock In {\em Computer algebra ({L}ondon, 1983)}, volume 162 of {\em
  Lecture Notes in Comput. Sci.}, pages 245--254. Springer, Berlin, 1983.

\bibitem{LLL}
A.~K. Lenstra, H.~W. Lenstra, Jr., and L.~Lov{\'a}sz.
\newblock Factoring polynomials with rational coefficients.
\newblock {\em Math. Ann.}, 261(4):515--534, 1982.

\bibitem{Lenstra}
H.~W. Lenstra, Jr.
\newblock Algorithms in algebraic number theory.
\newblock {\em Bull. Amer. Math. Soc. (N.S.)}, 26(2):211--244, 1992.

\bibitem{Levin1}
A.~Levin.
\newblock Siegel's theorem and the {S}hafarevich conjecture.
\newblock {\em J. Th\'eor. Nombres Bordeaux}, 24(3):705--727, 2012.

\bibitem{Liu2}
Q.~Liu.
\newblock {\em Algebraic geometry and arithmetic curves}, volume~6 of {\em
  Oxford Graduate Texts in Mathematics}.
\newblock Oxford University Press, Oxford, 2002.

\bibitem{Looijenga}
E.~Looijenga.
\newblock Compactifications defined by arrangements. {I}. {T}he ball quotient
  case.
\newblock {\em Duke Math. J.}, 118(1):151--187, 2003.

\bibitem{Lorenzotwists}
E.~Lorenzo.
\newblock Twists of non-hyperelliptic curves of genus $3$.
\newblock {\em arXiv:1604.02410}.

\bibitem{MartinDeschamps}
M.~Martin-Deschamps.
\newblock La construction de {K}odaira-{P}arshin.
\newblock {\em Ast\'erisque}, (127):261--273, 1985.
\newblock Seminar on arithmetic bundles: the Mordell conjecture (Paris,
  1983/84).

\bibitem{Moret-Bailly3}
L.~Moret-Bailly.
\newblock Hauteurs et classes de {C}hern sur les surfaces arithm\'etiques.
\newblock In {\em S\'eminaire sur les pinceaux de courbes elliptiques (\`a la
  recherche de Mordell effectif)}, volume 183 of {\em Soci\'et\'e de
  Math\'ematique de France Ast\'erisque}. 1990.

\bibitem{PoonenRat}
B.~Poonen.
\newblock Computing rational points on curves.
\newblock In {\em Number theory for the millennium, {III} ({U}rbana, {IL},
  2000)}, pages 149--172. A K Peters, Natick, MA, 2002.

\bibitem{Poo05}
B.~Poonen.
\newblock Varieties without extra automorphisms. {III}. {H}ypersurfaces.
\newblock {\em Finite Fields Appl.}, 11(2):230--268, 2005.

\bibitem{Remo}
G.~R{\'e}mond.
\newblock Hauteurs th\^eta et construction de {K}odaira.
\newblock {\em J. Number Theory}, 78(2):287--311, 1999.

\bibitem{Scholl}
A.~J. Scholl.
\newblock A finiteness theorem for del {P}ezzo surfaces over algebraic number
  fields.
\newblock {\em J. London Math. Soc. (2)}, 32(1):31--40, 1985.

\bibitem{Shaf1962}
I.~R. Shafarevich.
\newblock Algebraic number fields.
\newblock In {\em Proc. {I}nternat. {C}ongr. {M}athematicians ({S}tockholm,
  1962)}, pages 163--176. Inst. Mittag-Leffler, Djursholm, 1963.

\bibitem{Silverman}
J.~H. Silverman.
\newblock {\em The arithmetic of elliptic curves}, volume 106 of {\em Graduate
  Texts in Mathematics}.
\newblock Springer, Dordrecht, second edition, 2009.

\bibitem{Szpiroa}
L.~Szpiro, editor.
\newblock {\em S\'eminaire sur les pinceaux arithm\'etiques: la conjecture de
  {M}ordell}.
\newblock Soci\'et\'e Math\'ematique de France, Paris, 1985.
\newblock Papers from the seminar held at the {\'E}cole Normale Sup{\'e}rieure,
  Paris, 1983--84, Ast{\'e}risque No. 127 (1985).

\bibitem{Szpiro3}
L.~Szpiro.
\newblock Un peu d'effectivit\'e.
\newblock {\em Ast\'erisque}, (127):275--287, 1985.
\newblock Seminar on arithmetic bundles: the Mordell conjecture (Paris,
  1983/84).

\bibitem{Szpirob}
L.~Szpiro, editor.
\newblock {\em S\'eminaire sur {L}es {P}inceaux de {C}ourbes {E}lliptiques}.
\newblock Soci\'et\'e Math\'ematique de France, Paris, 1990.
\newblock {\`A} la recherche de ``Mordell effectif''. [In search of an
  ``effective Mordell''], Papers from the seminar held in Paris, 1988,
  Ast{\'e}risque No. 183 (1990).

\bibitem{Kanel}
R.~von K{\"a}nel.
\newblock An effective proof of the hyperelliptic {S}hafarevich conjecture.
\newblock {\em J. Th\'eor. Nombres Bordeaux}, 26(2):507--530, 2014.

\bibitem{Kanel2}
R.~von K{\"a}nel.
\newblock On {S}zpiro's discriminant conjecture.
\newblock {\em Int. Math. Res. Not. IMRN}, (16):4457--4491, 2014.

\bibitem{Zaal}
C.~Zaal.
\newblock Explicit complete curves in the moduli space of curves of genus
  three.
\newblock {\em Geom. Dedicata}, 56(2):185--196, 1995.

\end{thebibliography}
\bibliographystyle{plain}

\end{document}